\documentclass[12pt]{amsart}
\usepackage{amsmath,enumitem}
\usepackage[english,  activeacute]{babel}
\usepackage[latin1]{inputenc}
\usepackage{amssymb}
\usepackage{amsthm}
\usepackage{graphics}
\usepackage{array}
\usepackage{pdflscape}
\usepackage{a4wide}
\usepackage{color, url}
\usepackage{float}
\usepackage{graphicx}
\usepackage{hyperref}
\usepackage[dvipsnames]{xcolor}
\hypersetup{
	pdftoolbar=true,        
	pdfmenubar=true,        
	pdffitwindow=false,     
	pdfstartview={FitH},    
	pdftitle={},
	pdfauthor={},
	pdfsubject={},
	pdfkeywords={},
	pdfnewwindow=true,      
	colorlinks=true,       
	linkcolor=brown,          
	citecolor=brown,        
	filecolor=brown,      
	urlcolor=brown,           
	unicode=true          
}

\theoremstyle{plain}
\newtheorem{theorem}{Theorem}

\theoremstyle{definition}
\newtheorem{definition}[theorem]{Definition}
\newtheorem{example}[theorem]{Example}

\newcommand{\Z}{{\mathbb Z}}

\def\li{\text{\rm Li}}
\def\lif{\text{\rm Lif}}
\def\sts#1#2{\genfrac{\{}{\}}{0pt}{}{#1}{#2}}
\def\stf#1#2{\genfrac{[}{]}{0pt}{}{#1}{#2}}
\def\eul#1#2{\genfrac{\langle}{\rangle }{0pt}{}{#1}{#2}}
\def\blue#1{\textcolor{blue}{#1}}
\def\red#1{\textcolor{red}{#1}}
\title[Poly-Cauchy numbers - the combinatorics behind]{Poly-Cauchy numbers - the combinatorics behind}

\author{Be\'ata B\'enyi}
\address{\noindent Faculty of Water Sciences, University of Public Service, Baja, HUNGARY}
\email{benyi.beata@uni-nke.hu}

\author{Jos\'e L. Ram\'{\i}rez}
\address{\noindent Departamento de Matem\'aticas, Universidad Nacional de Colombia, Bogot\'a,  COLOMBIA}
\email{jlramirezr@unal.edu.co}

\date{\today}
\subjclass[2010]{05A05, 05A19}
\keywords{Poly-Cauchy numbers; permutations; poly-Bernoulli numbers.}

\begin{document}
\begin{abstract} We introduce poly-Cauchy permutations that are enumerated by the poly-Cauchy numbers. We provide combinatorial proofs for several identities involving poly-Cauchy numbers and some of their generalizations. The aim of this work is to demonstrate the power and beauty of the elementary combinatorial approach. 
\end{abstract}

\maketitle
\section{Introduction}
Poly-Cauchy numbers were defined by Komatsu \cite{Komc} motivated by the interesting properties of poly-Bernoulli numbers. While since the introduction many combinatorial interpretations and investigations appeared in the literature about poly-Bernoulli numbers, there is no such direction about the poly-Cauchy numbers yet. 

Let $n$ and $k$ be integers with $n\geq 0$ and $k \geq 1$. 
The \emph{poly-Cauchy numbers of the first kind} are defined by
\begin{align*}
c_{n}^{(k)}=n!\underbrace{\int_{0}^1 \cdots \int_{0}^1}_{k}\binom{t_1t_2\cdots t_k}{n}  \,dt_1dt_2\cdots dt_k,
\end{align*}
Moreover, the exponential generating function of $c_{n}^{(k)}$ is given by
\begin{align*}
\lif_k(\ln(1+t))=\sum_{n=0}^{\infty}c_n^{(k)}\frac{t^n}{n!},
\end{align*}
where
$\lif_k(t)=\sum_{n=0}^\infty\frac{t^n}{n!(n+1)^k}$
is the $k$-th \emph{polylogarithm factorial function}.  Similarly, the \emph{poly-Cauchy numbers of the second kind} are defined by
\begin{align*}
\hat{c}_{n}^{(k)}=n!\underbrace{\int_{0}^1 \cdots \int_{0}^1}_{k}\binom{-t_1t_2\cdots t_k}{n} \,dt_1\cdots dt_k,
\end{align*}
and the exponential generating function of $\hat{c}_{n}^{(k)}$ is given by
\begin{align*}
\lif_k(-\ln(1+t))=\sum_{n=0}^{\infty}\hat{c}_n^{(k)}\frac{t^n}{n!}, \quad (k\in \Z).
\end{align*}
Notice that the above generating functions have a meaning as a formal power series even if $k$ is non-positive. Therefore,  the poly Cauchy numbers of both kinds can be defined also for negative $k$.  Poly-Cauchy numbers of the first kind have the explicit formula
\begin{align*}
c_n^{(k)}= (-1)^{n}\sum_{m=0}^{n}\stf{n}{m} \frac{(-1)^m}{(m+1)^k},
\end{align*}
where $\stf{n}{m}$ denotes the unsigned Stirling numbers of the first kind that counts the number of permutations of $[n]:=\{1,2,\ldots, n\}$ into $m$ non-empty cycles, and can algebraically defined for instance by the following identity: 
\begin{align*}
\sum_{m=0}^n \stf{n}{m} x^m = x(x+1)\cdots (x+n-1).
\end{align*}
Poly-Cauchy numbers of the second kind have the explicit formula
\begin{align*}
\hat{c}_n^{(k)}= (-1)^{n}\sum_{m=0}^{n}\stf{n}{m} \frac{1}{(m+1)^k}.
\end{align*}
For more properties about these numbers see for example \cite{Cenk, Komc, Kom5, Kom4}.

In the literature, along the theory of poly-Bernoulli numbers and polynomials, the (analytical) theory of poly-Cauchy numbers were also developed, including several identities and generalizations (polynomials, $q$-parameters) etc. 

We are interested from the combinatorial point of view in poly-Cauchy numbers of the second kind with negative $k$-indices. In this case, these numbers are clearly integers. However, one notice that we have negative and positive integers as well, so actually $(-1)^n\hat{c}_n^{(-k)}$ are positive integers.  We give first a combinatorial interpretation for these numbers. For the sake of simplicity, let us denote these numbers by $\hat{c}_{n,k}$. Table \ref{tab1} shows the first few values of the sequence $\hat{c}_{n,k}$.

\begin{table}[h]
\begin{center}
\begin{tabular}{|c||c|c|c|c|c|c|}
\hline
$n \backslash k$
  & 0 & 1 & 2 & 3 & 4 & 5\\
\hline\hline
0 & 1 & 1& 1& 1 & 1 & 1\\
\hline
1 & 1 & 2 &4 & 8 & 16 & 32\\
\hline
2 & 2 & 5 & 13 & 35 & 97 & 275 \\
\hline
3 &6 & 17 & 51 & 161 & 531 & 1817 \\
\hline
4 & 24& 74 & 244 & 854 & 3148 & 12134\\
\hline
\end{tabular}
\end{center}
\caption{The first values of the sequence $\hat{c}_{n,k}$.}\label{tab1}
\end{table}

\section{Poly-Cauchy permutations}

In this section we introduce a new family of permutations enumerated by the sequence $\hat{c}_{n,k}$. First we recall some well-known facts about permutations. Permutations can be defined and seen from different point of views, as a map, as an arrangement, as a product of cycles, etc. It is useful to switch between these point of views, however it is always important to be conscious of the particular approach that is used in an argument. In our case cycles will play the crucial role. 

Any permutation is the product of cycles, and cycles can be recorded uniquely: we write the largest element of each cycle first and we arrange the cycles in increasing order of their first elements. This representation is called \emph{canonical cycle notation}. 

On the other hand, considering a permutation as a map, $\pi:[n]\rightarrow [n]$, $\pi_i=\pi(i)$ each permutation can be written in one-line notation, 
$\pi = \pi_1\pi_2 \cdots \pi_n$. There is a well-known bijection between permutations of $[n]$ written in cycle notation and written in one-line-notation, which we recall now. We say that $\pi_i$ is a \emph{left-to-right maximum} if, for all $k < i$, we have $\pi_k < \pi_i$. For example, the permutation $3261457$ has three left-to-right maxima. These are the entries
$3, 6$, and $7$.  Let $\pi$ be a permutation of $n$ elements  written in canonical cycle notation, and let $\pi '$ be the permutation written in the one-line notation that is obtained from $\pi$ by omitting all parentheses. Notice that the canonical cycle notation can be obtained from the left-to-right maxima of the permutation. 

\begin{definition}
Let $n$ and $k$ be non negative integers. A \emph{$(n,k)$-poly-Cauchy permutation} is a permutation of $[n+k]$ with the following properties:
\begin{itemize}
\item[(A1)] each maximal contiguous substring whose support belongs to $\{n+1,n+2,\ldots,n+k\}$ is increasing, 
\item[(A2)] the left-to-right maxima of the maximal contiguous substrings whose support belongs to  $\{1,2,\ldots, n\}$ are increasing from left to right, when ``reading'' the whole permutation. 
\end{itemize}
\end{definition}
\begin{example}\label{example_first}
A $(9,6)$-poly-Cauchy permutation is \begin{align*}
\pi= 11-13-6-1-2-15-7-4-8-10-12-14-9-5-3.
 \end{align*}
\end{example}

We introduce now some terminology and notation. Let $\mathcal{P}_{n,k}$ denote the set of all $(n,k)$-poly-Cauchy permutations. For simplicity of notation we use the notation of red elements for $\{\red{1}, \red{2},\ldots, \red{n}\}$ and the notation of  blue elements for $\{\blue{1},\blue{2},\ldots, \blue{k}\}$ instead of the set $\{n+1,n+2,\ldots,n+k\}$ (which is essentially the same). Given an $(n,k)$-poly-Cauchy permutation $\pi\in \mathcal{P}_{n,k}$ we will denote the arrangement of the elements $\{\red{1},\ldots, \red{n}\}$, seeing itself as a permutation of $[n]$ as $\red{\pi^n}$. Similarly, the permutation/or ordered partition on the elements $\{\blue{1},\ldots, \blue{k}\}$ as $\blue{\pi^k}$.

\begin{example}
The permutation given in Example \ref{example_first} is in this notation $\blue{24}\red{612}\blue{6}\red{748}\blue{135}\red{953}$.
Maximal contiguous is the whole sequence of red elements and blue elements, respectively. In the first permutation the blue sequences are $\blue{24}$, $\blue{6}$, $\blue{135}$, and the red sequences: $\red{\textbf{6}12}$, $\red{\textbf{7}4\textbf{8}}$, $\red{\textbf{9}53}$. In the blue sequences are the elements increasingly ordered and the bold elements in the red sequences denote the left-to-right maxima. $\red{\pi^n}=\red{612748953}$ and $\blue{\pi^k}=\blue{246135}$. 
\end{example}

Note that the merging $\red{\pi^n}$ and $\blue{\pi^k}$ happens so that a blue sequence is always inserted before a left-to-right maxima of $\red{\pi^n}$. On the other hand, a red sequence is always inserted after a descent top of $\blue{\pi^k}$. (A \emph{descent top} of a permutation $\sigma$ is $\sigma_i$, if $\sigma_{i}>\sigma_{i+1}$, i.e., an element followed by a smaller element.)

As a visualization of a poly-Cauchy permutation we can use permutation matrices. Recall that for a permutation $\pi=\pi_1\pi_2\cdots \pi_n$, its permutation matrix is obtained  by placing a dot in columns $i$ and row  $\pi_i$ from below. The poly-Cauchy permutations can be represented by a permutation  matrix as in Figure \ref{mat}.  

\begin{figure}[ht]
\begin{center}
 \includegraphics[scale=0.7]{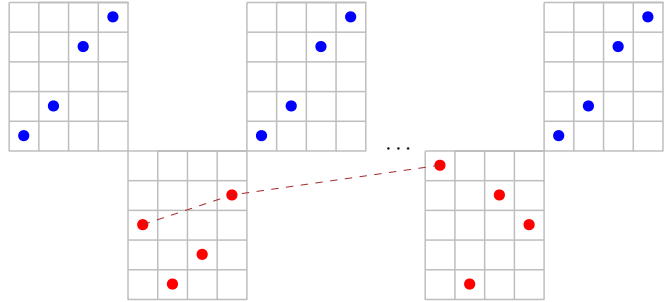}
 \caption{Decomposition of a poly-Cauchy permutation.}\label{mat}
\end{center}
\end{figure}

\begin{example}
For $n=2$ and $k=1$ we have the $5$ permutations:
\begin{align*}
\blue{1}\red{21}, \quad \blue{1}\red{12}, \quad \red{21}\blue{1}, \quad \red{12}\blue{1},\quad  \red{1}\blue{1}\red{2}. 
\end{align*}
For $n=2$ and $k=2$ we have the $13$ permutations:
\begin{align*}
&\blue{12}\red{12},\quad \blue{12}\red{21}, \quad \blue{1}\red{12}\blue{2},\quad  \blue{1}\red{21}\blue{2},\quad \blue{2}\red{12}\blue{1},\quad  \blue{2}\red{21}\blue{1},\quad \blue{1}\red{1}\blue{2}\red{2},\quad \blue{2}\red{1}\blue{1}\red{2} \\
&\red{12}\blue{12},\quad \red{21}\blue{12},\quad \red{1}\blue{12}\red{2},\quad \red{1}\blue{1}\red{2}\blue{2},\quad \red{1}\blue{2}\red{2}\blue{1}.
\end{align*}
\end{example}

\begin{theorem}
The $(n,k)$-poly-Cauchy permutations are enumerated by the sequence $\hat{c}_{n,k}$.
\end{theorem}
\begin{proof}
The  $(n,k)$-poly-Cauchy permutations  can be constructed by the following procedure: 
\begin{enumerate}
\item First, permute the elements $\{\red{1},\red{2},\ldots, \red{n}\}$, and write it into canonical cycle notation. This step will ensure the condition (A2).
\item Secondly,  insert the elements of $\{\blue{1},\ldots, \blue{k}\}$ in between the cycles (before or after). Notice that we have $m+1$ spaces and some of them may be left empty. The elements in between the cycles are written in increasing order. This step will ensure the condition (A1).
\end{enumerate}
This way we obtain an alternating sequence of cycles of $\{\red{1},\ldots, \red{n}\}$ and blocks of $\{\blue{1},\ldots, \blue{k}\}$.
It is clear that the number of $(n,k)$-poly-Cauchy permutations constructed by this procedure is given by
\begin{align*}
\hat{c}_{n,k}= \sum_{m=0}^{n}\stf{n}{m}(m+1)^k.
\end{align*}
In Figure \ref{ex1} we give an example of this procedure for a $(9,6)$-poly-Cauchy permutation. 
\begin{figure}[H]
\begin{center}
 \includegraphics[scale=0.6]{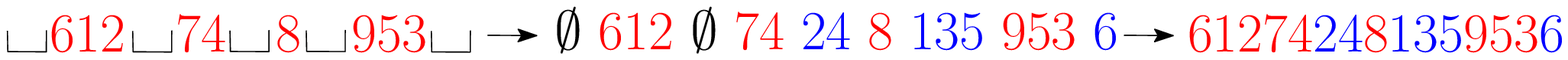}
 \caption{Procedure to construct a poly-Cauchy permutation.}\label{ex1}
\end{center}
\end{figure}
\end{proof}

We give some examples for some particular values of $n$ and $k$. 
\begin{example}
For $k=0$ we have $\hat{c}_{n,0}=n!$, since $(n,0)$-poly-Cauchy permutation is simple a permutation of $[n]$. 
\end{example}
\begin{example}
For $n=1$ we have $\hat{c}_{1,k}=2^k$. To obtain a $(1,k)$-poly-Cauchy permutation we  insert $\red{1}$ in between two sets of the blue elements (before or after),  $\{\blue{1},\blue{2},\ldots, \blue{k}\}$. So, $|\mathcal{P}_{1,k}|$ is the number of choosing $i$ elements from $\{\blue{1},\blue{2},\ldots, \blue{k}\}$ that are to the left of $\red{1}$, allowing $i$ to be $0$ and $k$. 
\end{example}
\begin{example}
For $n=2$ we have $\hat{c}_{2,k}=2^k + 3^k$. In this case, we have the sets $\{\red{1}, \red{2}\}$ and $\{\blue{1},\blue{2},\ldots, \blue{k}\}$, so the number of ways we get such a permutation is to insert the string $\red{21}$ in between two sets of the blue elements, before or after, so we have $2^k$ ways. We can also insert the strings  $\red{1}$  and $\red{2}$ in between two sets (possibly empty), so we have $3^k$ ways. 
\end{example}
In the following theorems we will give combinatorial arguments of some identities involving  poly-Cauchy numbers, poly-Bernoulli numbers, Stirling numbers of both kinds, and Eulerian numbers. 

\begin{theorem}\label{theo:orthogonal}
For non-negative $n$ and $k$ it holds
\begin{align*}
	\sum_{m=0}^{n}(-1)^m\sts{n}{m}\hat{c}_{m,k} = (-1)^n(n+1)^k.
\end{align*}
\end{theorem}
\begin{proof}
The right hand side is (besides the $(-1)$ factor) a poly-Cauchy permutation where the elements $\{\red{1},  \ldots, \red{n}\}$ are in increasing order (with some blocks of the elements $\{\blue{1}, \ldots, \blue{k}\}$ in between, before or after them), $\red{\pi^n}$ is the identity permutation. We call such a permutation \emph{id-poly-Cauchy permutation}.
	
We define an involution, $\phi$, to show that the left hand side counts id-poly-Cauchy permutations of size $(n,k)$. The involution $\phi$ is defined on the set of poly-Cauchy permutations with an extra structure on the red elements. Partition the elements $\{\red{1}, \red{2},  \ldots, \red{n}\}$  into $m$ non-empty blocks (write the elements of a block into decreasing order, and consider the greatest element of a block as its ``representative''). Then permute these blocks into a permutation with $j$ cycles, $j\leq m \leq n$. Insert between the cycles of blocks the blue elements, $\{\blue{1}, \blue{2} \ldots, \blue{k}\}$. in order to obtain a $(n,k)$-poly-Cauchy permutation with an extra structure on the substrings of red elements (the blocks). Note $j$ is the number of left-to-right maxima among the elements $\{\red{1}, \red{2},  \ldots, \red{n}\}$.  

For example, one set partition of $\{\red{1}, \red{2},  \ldots, \red{9}\}$ into $m=5$ blocks is $B_1=\{\red{4}, \red{2}, \red{1}\}, B_2=\{\red{5}, \red{3}\}, B_3=\{\red{6}\}, B_4=\{\red{8}, \red{7}\}, B_5=\{\red{9}\}$. Now, we permute these five blocks  to obtain  a permutation in $j=3$ cycles. For example, the permutation $(32)(4)(51)$ corresponds to the $(B_3,B_2)(B_4)(B_5,B_1)$. Finally, we insert the blue blocks, for example we can insert the blocks $ \{\blue{1}, \blue{3}\},  \{\blue{2}, \blue{4}, \blue{5}\}, \{\blue{6}\},$ and we obtain the $(9,6)$-poly-Cauchy permutation $\blue{13}\red{653}\red{87}\blue{245}\red{9421}\blue{6}$.

If the number of cycles is $n$, that is $j=n$, that can happen only one way, every block contains only one element and each cycle one block, hence, after inserting the blocks of $\{\blue{1}, \blue{2} \ldots, \blue{k}\}$,  we obtain an id-poly-Cauchy permutations.
	
Otherwise, i.e., if the number of cycle is less than $n$, then there is at least one cycle that contains at least two elements. Let $a$ be the greatest element which is not alone in a cycle. 

If $a$ is in a block alone, merge it with the block that follows it in the cycle. If $a$ is not alone separate it into a block to be alone.  
	
This way we defined an involution, $\phi$,  between the set of poly-Cauchy permutations with extra structure with even number of left-to-right maxima in $\red{\pi^n}$ and odd number of left-to-right maxima in $\red{\pi^n}$.
 
\end{proof}
Many formulas given in the literature are consequences of Theorem \ref{theo:orthogonal} and modified forms of the involution $\phi$ can be used to provide combinatorial proofs for them.  For instance, Komatsu presents a formula in \cite{Kom6} (Theorem 3) that can be seen as a generalization of Theorem \ref{theo:orthogonal}. Let $\sts{n}{k}_r$ denote the $r$-Stirling numbers of the second kind that count for instance the number of partitions of $[n]$ into $k$ non-empty blocks such that $\{1,2,\ldots,r\}$ are in distinct blocks. These numbers were introduced by Broder \cite{Broder}.
\begin{theorem}[\cite{Kom6}, Theorem 3] 
For non-negative $n$, $k$ and $r$, with $r\leq n$, it holds
\begin{align}\label{form:r-orth}
\sum_{j=r}^{n}(-1)^{j}\sts{n}{j}_r\hat{c}_{j,k}  = \sum_{\ell =1}^r \stf{r}{\ell} (n-r+\ell +1)^k.
\end{align}
\end{theorem}
\begin{proof}
The idea of the proof is the same as in the proof of Theorem \ref{theo:orthogonal}. Partition the elements $\{\red{1},\ldots, \red{n}\}$ into $j$ non-empty blocks such that the elements $\{\red{1},\ldots, \red{r}\}$ are in distinct blocks. As usual, we call these elements \emph{special elements} and the remaining elements, $\{\red{r+1},\ldots, \red{n}\}$, as \emph{non-special elements}. Arrange the blocks into $i$ cycles ($i\leq j\leq n$). Merge the so obtained permutation with blocks of $\{\blue{1}, \blue{2} \ldots, \blue{k}\}$ in order to obtain a $(n,k)$-poly-Cauchy permutation with an extra structure on the substrings of red elements (the blocks). As before, if the number of cycles is $n$, every block contains only one element, and each cycle one block, hence after merging with the blocks  of $\{\blue{1}, \blue{2} \ldots, \blue{k}\}$  we obtain an id-poly-Cauchy permutations. 

If $j\not=n$, there is at least one cycle that contains at least two elements. Note that the involution $\phi$, we defined in the proof of Theorem \ref{theo:orthogonal} can not be used in that form, since we can not merge blocks containing the special elements. We will modify slightly the definition of $a$, the base element of our involution, $\phi '$. Let $a$ be the greatest \textit{non-special} element which is not alone in a cycle. 

If such an $a$ exists, we can apply the involution $\phi$ based on this element, i.e., if $a$ is in a block alone, merge it with the block that follows it in the cycle. If $a$ is not alone separate it into a block to be alone. 

However, such an $a$ does not have to exist. If such an $a$ does not exist, it means, all the non-special elements are in a distinct cycle, (and so distinct blocks), while the special elements are in distinct blocks by definition but not necessarily in distinct cycles. $\red{\pi^n}$ is a permutation of the form: $\red{\pi^n}=(c_1)(c_2)\cdots (c_{\ell})(r+1)(r+2)\cdots(n)$, where $\ell$ is the number of cycles into the special elements are ordered. Note that these cycles contain only special elements (otherwise $a$ with the required property would exist). The blue elements are inserted in the poly-Cauchy permutation in between these cycles of $\red{\pi^n}$ (before or after).

The involution $\phi '$ shows that the left hand side of the identity \eqref{form:r-orth} counts such poly-Cauchy permutations.

On the other hand, such poly-Cauchy permutations can be constructed by the following procedure. First arrange the $r$ special elements into $\ell$ cycles, and take the non-special elements as fix points. Record this permutation in canonical cycle notation and insert the blue elements $[k]$ in between, before or after the cycles decoding this insertion by a word $w=w_1w_2\cdots w_k$, where $w_i$ is the number of cycles to the left of the element $i$. The number of ways of doing this is clearly given by the formula of the right hand side of the Identity (\ref{form:r-orth}).
\end{proof}

Theorem \ref{theo:orthogonal} implies also the formula that connects poly-Cauchy numbers and poly-Bernoulli numbers \cite{Kom7}.  The \emph{poly--Bernoulli numbers},  denoted by $B_{n}^{(k)}$, where $n$ is a positive integer and $k$ is an
integer, are defined by the following exponential generating function \cite{Kaneko}
\begin{equation}\label{gefunpoly}
\sum_{n=0}^{\infty}B_n^{(k)}\frac{x^n}{n!} =\frac{\li_k(1-e^{-x})}{1-e^{-x}},
\end{equation} where $\li_k(z) =\sum_{i=1}^{\infty}\frac{z^i}{i^k}$ is the $k$-th polylogarithm function. Note that for $k=1$ we recover the classical Bernoulli numbers, that is,  $B_n^{(1)}=(-1)^nB_n$, for $n\geq 0$, where $B_n$ denotes the $n$-th Bernoulli number. Poly-Bernoulli numbers were introduced by Kaneko in 1997 \cite{Kaneko} as he noticed that the generating function of the usual Bernoulli numbers  can be generalized by using the polylogarithm function. From the combinatorial point of view the array with negative $k$ indices are interesting, since these numbers are integers (see sequence  A099594 in \cite{OEIS}).  Table \ref{tab2} shows the first few values of the sequence $B_n^{(k)}$.
\begin{table}[h]
\begin{center}
\begin{tabular}{|c||c|c|c|c|c|c|}\hline
$k \backslash n$ & 0 & 1 & 2 & 3& 4& 5\\ \hline\hline
0 & 1 & 1 & 1 & 1 & 1 & 1 \\ \hline
-1 & 1 & 2 & 4 & 8 & 16 & 32\\\hline
-2 & 1 & 4 & 14 & 46 & 146 & 454\\\hline
-3 & 1 & 8 & 46 & 230 & 1066 & 4718\\\hline
-4 & 1 & 16 & 146 & 1066 & 6902 & 41506\\\hline
-5 & 1 & 32 & 454 & 4718 & 41506 & 329462\\ \hline
\end{tabular}
\end{center}
\caption{The first values of the sequence $B_n^{(k)}$.}\label{tab2}
\end{table} 

The poly-Bernoulli numbers count several combinatorial objects (see for example \cite{Benyi, BH, Nagy, BR, Brewbaker}),   in particular count Callan permutations \cite{BH1}.  An \emph{$(n,k)$--Callan permutation} is a permutation of $[n+k]$ such that each maximal contiguous substring whose support belongs to $\{1,2,\ldots, n\}$ (respectively $\{n+1,n+2,\ldots, n+k\}$) is increasing (respectively decreasing). It is easy to see that an $(n,k)$-Callan permutation can be seen as an alternating sequence of blocks of partitions of $\{\red{1},\ldots, \red{n}\}$ and $\{\blue{1},\blue{2},\ldots, \blue{k}\}$. 
		
\begin{theorem}[\cite{Kom7}, Theorem 3.2] 
For non-negative $n$ and $k$ it holds
\begin{align*}
	B_n^{(-k)} = (-1)^n \sum_{\ell = 1}^{n}\sum_{m=1}^{n} m! \sts{n}{m} \sts{m}{\ell} \hat{c}_{\ell}^{(-k)}.	
\end{align*}
\end{theorem}
\begin{proof}
Let us rewrite the identity as follow:
\begin{align*}
		B_n^{(-k)} = (-1)^n \sum_{m=1}^{n} m! \sts{n}{m}\sum_{\ell = 1}^{m}(-1)^{\ell}\sts{m}{\ell} \hat{c}_{\ell, k}.		\end{align*}
 By Theorem \ref{theo:orthogonal} the inner sum equals to $(-1)^m(m+1)^k$, and we obtain the classical formula of poly-Bernoulli numbers:
 \begin{align*}
 	B_n^{(-k)} = \sum_{m=1}^{n} m! \sts{n}{m} (-1)^{n+m}(m+1)^k.
 \end{align*}
 Combinatorially, by the proof of the Theorem \ref{theo:orthogonal} after applying the involution, $\phi$, the inner sum $\sum_{\ell = 1}^{m}(-1)^{\ell}\sts{m}{\ell} \hat{c}_{\ell, k}$ counts $(m,k)$-id-poly-Cauchy permutations. Now let $A_m$ be the set of Callan permutations with at most $m$ different red blocks (blocks of elements $\{\red{1},\ldots, \red{n}\}$). We obtain from an $(m,k)$-id-poly-Cauchy permutation an $(n,k)$-Callan permutation from $A_m$ as follows. Given an id-poly-Cauchy permutation, permute first the $m$ red elements in $m!$ ways. Then partition $\{\red{1},\ldots, \red{n}\}$ into $m$ non-empty blocks, $\red{B_1},\red{B_2},\ldots,\red{B_m}$, and replace the red element $i$ in the poly-Cauchy permutation by a block $\red{B_i}$. Applying the inclusion-exclusion principle completes the proof. 
\end{proof}

The other direction can be seen similarly combinatorially.
\begin{theorem}[\cite{Kom7}, Theorem 3.1]
For non-negative $n$ and $k$ it holds 
	\begin{align*}
		\hat{c}_n^{(-k)} = (-1)^n\sum_{\ell = 1}^n \sum_{m=1}^n \frac{1}{m!} \stf{n}{m}\stf{m}{\ell}B_{\ell}^{(-k)}.
	\end{align*}
\end{theorem}
\begin{proof}
	Take now a $(\ell, k)$-Callan permutation and a permutation of $[m]$ into $\ell$ cycles. Replace each red element in the Callan permutation by a cycle. If there are more cycles in a block take the product of that cycles and record this way $m$ elements in between the blue blocks of the Callan permutation. (So, the $m$ elements are not in cycles anymore, but create subsequences of a permutation of $m$ with places in between.) (Note that if we ignore the blue blocks, the red elements determine a permutation of $m$, and summing up on all $\ell$, every permutation will occur.) 	Order the elements $\{\red{1},\red{2},\ldots, \red{n}\}$ into $m$ cycles and replace each red element in our construction by a cycle in the canonical order (first the cycle with the least greatest element and so on).	This way we obtain a poly-Cauchy permutation.     
\end{proof}
\begin{example}
	We show the procedure in the proof along an example. Let $n=9$, $m=5$, $\ell=3$, $k=6$. Let the $(\ell,k)$-Callan permutation be $\blue{6}\red{23}\blue{135}\red{1}\blue{24}$. The permutation of $[m]$ into $\ell$ cycles $(1)(43)(52)$, and the permutation of $[n]$ into $m$ cycles $(43)(5)(71)(82)(96)$. 
	\begin{align*}
		&\blue{6}\red{23}\blue{135}\red{1}\blue{24}\quad\rightarrow\quad \blue{6}\red{(43)(52)}\blue{135}\red{(1)}\blue{24}\quad\rightarrow\quad \blue{6}\red{5432}\blue{135}\red{1}\blue{24}\\
		&\blue{6}\red{\emptyset\emptyset\emptyset\emptyset}\blue{135}\red{\emptyset}\blue{24}\quad\rightarrow\quad \blue{6}\red{4357182}\blue{135}\red{96}\blue{24}.
	\end{align*}
\end{example}

Next we provide a combinatorial proof for the closed formula of poly-Cauchy numbers. 
\begin{theorem}\label{theo:closed}
For non-negative $n$ and $k$ it holds 
\begin{align}\label{form:closed}
	\hat{c}_{n,k} = \sum_{j=0}^{k}j!\stf{n+1}{j+1}\sts{k+1}{j+1}.
\end{align}
\end{theorem}

\begin{proof}
We construct the $(n,k)$-poly-Cauchy permutations as follows. Let $\overline{K}$ be  the set $[k]$ extended by a special element $\blue{*}$. First partition the set $\overline{K}$ into $j+1$ non-empty blocks.  Second, order the set $[n]$ into $\ell$ cycles in the canonical cycle notation, with $j\leq \ell \leq n$.  Then we choose $j$ of these cycles  to insert  a blue ordinary block (blocks without the element $\blue{*}$) directly before each cycle, so we have $\ell(\ell-1)\cdots (\ell-j+1)=\binom{\ell}{j}j!$ options to insert and organize these blocks. Finally, the extra block, that contains $\blue{*}$ is inserted after the last red cycle, at the end of the arrangement, and we delete the element $\blue{*}$. So if the extra block contains only $\blue{*}$ our poly-Cauchy permutation ends with a red element, if it contains also other elements, it ends with a blue element.  Therefore we have:
\begin{align*}
\hat{c}_{n,k} = \sum_{j=0}^{k}\sts{k+1}{j+1}\sum_{\ell = j}^{n}\stf{n}{\ell}\binom{\ell}{j}j!= \sum_{j=0}^{k}j!\sts{k+1}{j+1}\stf{n+1}{j+1}.
\end{align*}
Notice that in the last equality we use the  identity $\sum_{\ell = j}^{n}\binom{\ell}{j}\stf{n}{\ell} = \stf{n+1}{j+1}$.
\end{proof}

In the following two theorems we give recurrence relations to calculate the sequence $\hat{c}_{n,k}$.
\begin{theorem}\label{teorec}
For $n\geq 1$ and $k\geq 0$ we have
\begin{align*}
	\hat{c}_{n,k} = (n-1)\hat{c}_{n-1,k} +\sum_{i=0}^{k}\binom{k}{i}\hat{c}_{n-1,k-i}.
\end{align*}
\end{theorem}

\begin{proof}
The greatest red element, $n$, is surely in the last cycle of the permutation $\red{\pi^n}$, but we consider the cases when it is the very last red element and when it is followed by some other red elements separately. 
An $(n,k)$-poly-Cauchy permutation where $n$ is the last red element can be obtained the following way. First we choose $i$ blue elements ($i$ can be empty) in a block say $\blue{B}$. We take a $(n-1,k-i)$-poly-Cauchy permutation and extend it by the sequence of $\red{n}\blue{B}$ at the end. This can be done in $\sum_{i=0}^{k}\binom{k}{i}\hat{c}_{n-1,k-i}$ ways.

In order to show that $(n,k)$-poly-Cauchy permutations where $n$ is not the last red element is  $(n-1)\hat{c}_{n-1,k} $ we establish a bijection. 

Given an $(n-1,k)$-poly-Cauchy permutation $\pi$ and a value  $1\leq i\leq n-1$ we construct an $(n,k)$-poly-Cauchy permutation $\overline{\pi}$ where $n$ is not the last element as follows. $\pi$ is a merging of $\pi^{n-1}$ and $\pi^k$. Consider $\pi^{n-1}$ as the product of cycles. Insert $n$ before the element $i$ in its cycle, to obtain $\overline{\pi}^n$. Merge now $\overline{\pi}^n$ and $\pi^k$ again together. Note that we did not change the places of where the blue elements are inserted, but we possible changed the relative order of the red cycles, since the cycle containing $9$ we be shifted to the end, and other cycles may shifted to the left by one. Since each step can be reversed, this procedure is a bijection, which implies our statement.
\end{proof}
\begin{example}
We give an example for making the bijection in the previous proof clearer.
Let $\pi=\blue{6}\red{43}\blue{24}\red{561}\blue{135}\red{827}$ and $i=6$. The cycles of $\pi^8$ are $(\red{43})$, $(\red{5})$, $(\red{61})$ and $(\red{827})$. We insert $\red{9}$ after $\red{6}$ and obtain $(\red{61})\rightarrow(\red{691})=(\red{916})$. After merging the cycles and blocks again we obtain $\overline{\pi}=\blue{6}\red{43}\blue{24}\red{5827}\blue{135}\red{916}$.  
\end{example}
\begin{theorem}
For $n\geq 1$ and $k\geq 0$ we have
\begin{align*}
	\hat{c}_{n,k} = \sum_{i=0}^{n}\sum_{j=0}^k (n-1)_i\binom{k}{j} \hat{c}_{n-1-i,k-j}.
\end{align*}
\end{theorem}
\begin{proof}
	Let $j$ be the number of the last blue block (possibly empty) and $i$ be the number of elements that the last red cycle contain besides the element $n$. (In the standard cycle notation the last cycle necessarily contains $n$.) We choose the elements for the block in $\binom{k}{j}$ ways and construct the last cycle by arranging $i$ elements after $n$ in $\binom{n-1}{i}i!=(n-1)_i$ ways. From the remaining elements we construct a poly-Cauchy permutation. 
\end{proof}

The \emph{Eulerian numbers} $\eul{n}{k}$ counts the number of permutations $\pi=\pi_1\pi_2\cdots \pi_n$  with  $k-1$ descents, that is  $k-1=|\{i\in[n-1]: \pi_i>\pi_{i+1} \}|$. Note that a permutation $\pi$ of $[n]$ with $k-1$ descents is the union of $k$ increasing subsequences of consecutive entries,
also called \emph{ascending runs}.  For example, if $\pi=258193647$, then $3, 5$, and $7$ are the descents. Moreover, $\pi$ is the union of the ascending runs: $258, 19, 36$, and $47$.   In Theorem \ref{Euler} we  give a combinatorial identity involving the Eulerian numbers and the sequence $\hat{c}_{n,k}$.

\begin{theorem}\label{Euler}
For non-negative $n$ and $k$ it holds 
\begin{align*}
	\hat{c}_{n,k} = \sum_{m=0}^n\sum_{i=0}^m \binom{k-i}{m-i}\eul{k}{i}\stf{n+1}{m+1}.
\end{align*}
\end{theorem}
\begin{proof}
We turn now our point of view a little over and consider the construction of a poly-Cauchy permutation as inserting red cycles into a permutation of blue elements, $\blue{\pi^k}$. The properties of poly-Cauchy permutations require that the ascending runs of $\pi^k$ are separated by a red cycle. But it is not forbidden to insert cycles after an ascent, cutting ascending runs this way into more pieces. Hence, the $(n,k)$-poly-Cauchy permutations can be constructed by the following process.  Take a permutation of $[k]$ with $i$ ascending runs and mark $m-i$ ascents. Permute the red elements into at least $m$ cycles, say $\ell$, and choose $m$ out of them. Merge the two permutation in order to obtain a poly-Cauchy permutation by inserting a chosen cycle after each ascending runs and marked ascents. 
\end{proof}

\section{Poly-Cauchy polynomials of the second kind}
This section introduces a combinatorial interpretation for the poly-Cauchy polynomials. 
The poly-Cauchy polynomials of the second kind, denoted by $\hat{c}_{n}^k(z)$, where introduced by Kamano  and  Komatsu \cite{KanKom} by using the following analytical formula 
\begin{align*}
\hat{c}_{n}^{(k)}(z):=n!\underbrace{\int_{0}^1 \cdots \int_{0}^1}_{k}\binom{-x_1x_2\cdots x_k-z}{n} \,dx_1 \,dx_2\cdots dx_k.
\end{align*}
The above definition is equivalent to the explicit formula
\begin{align*}
\hat{c}_{n}^k(z)= (-1)^{n}\sum_{m=0}^{n}\stf{n}{m} \sum_{i=0}^m\binom mi \frac{(-z)^i}{(m-i+1)^k}.
\end{align*}

We consider poly-Cauchy polynomials of the second kind without negative $k$-indices and ignore the negative sign, i.e.,

\begin{align}\label{form:pCpol_basic}
	\hat{c}_{n,k}(z):= \sum_{m=0}^{n}\stf{n}{m} \sum_{i=0}^m\binom mi (-1)^i(m-i+1)^kz^i
\end{align}
We define \emph{partial poly-Cauchy permutations} as a pair $P=(\pi, \sigma)$, where the reduction of $\pi$ on the red elements is a partial permutation of $[n]$ (not all elements of $[n]$ are included) and $\sigma$ is a permutation of the remaining elements of $[n]$. 
Further, we define the weight of a partial poly-Cauchy permutation, $w(P)=w(\pi, \sigma)$ as the number of cycles in the permutation $\sigma$. 
\begin{example}
	The weight of the partial poly-Cauchy permutation $(\blue{24}\red{61}\blue{6}\red{78}\blue{135}\red{93}, \red{(24)(5)})$ is $2$.
\end{example}

Let $\mathcal{PP}_{n,k}$ denote the set of partial poly-Cauchy permutations.
\begin{theorem} We have
\begin{align*}
	\sum_{P\in \mathcal{PP}_{n,k}}(-z)^{w(P)}=\hat{c}_{n,k}(z). 
\end{align*}
\end{theorem}
\begin{proof}
 The sum is clearly equal to the formula \eqref{form:pCpol_basic} according to our previous considerations.
\end{proof}
Recall the well-known relation between the falling factorial $(x)_n=x(x-1)\cdots(x-n+1)$ and the (signed) Stirling number of the first kind, $s(n,k)$: 
\begin{align*}
	(x)_n = \sum_{k=0}^n s(n,k)x^k.
\end{align*}
From the combinatorial point of view this identity expresses that if we weight a permutation $\pi$ with the number of its cycles, i.e., $w(\pi)=\#\, \mbox{of cycles}$, we have 
\[(-1)^n(x)_n=\sum_{\pi\in S_n}(-x)^{w(\pi)},\]
where $S_n$ denotes as usual the set of permutations of $[n]$. 
The following theorem is based on this classical result. 
\begin{theorem}[\label{theo:pol_conv}\cite{KimKim}, Theorem 7]
For non-negative $n$ and $k$ it holds
	\begin{align}\label{form:pol_1}
		\hat{c}_{n}^{(-k)}(z)=\sum_{m=0}^n\binom{n}{m}\hat{c}_{n-m}^{(-k)}(z)_m.
\end{align}
\end{theorem}
\begin{proof}
Let $m$ be the number of elements that are contained in $\sigma$ in the partial poly-Cauchy permutation $P=(\pi, \sigma)$. Choose these $m$ elements in $\binom{n}{m}$ ways and construct $\pi$ from the remaining elements in $\hat{c}_{n-m,k}$ ways. Hence, we have
 \begin{align*}
\sum_{P\in \mathcal{PP}_{n,k}}(-z)^{w(P)}=\sum_{m=0}^n\binom{n}{m}\hat{c}_{n-m,k}(-1)^m(z)_m.
\end{align*} 
Since $\hat{c}_{n-m,k}=(-1)^{n-m}\hat{c}_{n-m}^{(-k)}$, the theorem follows. 
\end{proof}
The theorems proven in the previous section can be generalized easily to partial poly-Cauchy permutations, and hence to poly-Cauchy polynomials. As an example we present the combinatorial proof for the generalization of Theorem \ref{theo:orthogonal}. 
\begin{theorem}[\cite{KanKom}, Theorem 4]
For non-negative $n$ and $k$ it holds	
\begin{align*}
	\sum_{m=0}^{n}(-1)^m\sts{n}{m}\hat{c}_{m,k}(z) =\sum_{i=0}^{n}\binom{n}{i}(-1)^iz^i(n-i+1)^k.
\end{align*}
\end{theorem}
\begin{proof}
	Applying the involution $\phi$ of the proof of Theorem \ref{theo:orthogonal} on the set of partial poly-Cauchy permutations, $\mathcal{PP}$, we recognize that only such partial poly-Cauchy permutations do not cancel in that the red elements are in increasing order. This means now that $\sigma$ is also fixed as an identity permutation on the elements that are contained there. So the weight of $\sigma$ is simply $(-z)^i$, if $i$ is the size of $\sigma$. 
\end{proof}
As we see the proof is straightforward, so we omit the proof of other generalizations. Instead, we show combinatorially two identities that are typical for polynomials. Moreover, in these proofs we introduce another combinatorial interpretation of the poly-Cauchy polynomials (which arises also very naturally). 

Looking at the formula \eqref{form:pol_1} we can give the combinatorial interpretation for $z$ positive integers as follow. Consider now a poly-Cauchy permutation with $z$ ordered boxes. Choose some elements of $[n]$ and put them into one of the boxes and form from the remaining elements a poly-Cauchy permutation. We call such a construction an \emph{extended poly-Cauchy permutations}. Having this picture in mind, the following two identities are immediate. 
\begin{theorem}[\cite{KimKim}]
We have
\begin{align}
	\hat{c}_{n,k}(x+1) = \hat{c}_{n,k}(x) + n\hat{c}_{n-1,k}(x).
\end{align}
\end{theorem}
\begin{proof}
	The left hand side counts extended poly-Cauchy permutations with $x+1$ boxes. There are two cases, the last box contains an element, or not. If the last box is empty, we can delete it and the number of such extended permutations is just $ \hat{c}_{n,k}(x)$. If the last box contains any of the $n$ elements (that can be chosen in $n$ ways) the remaining $n-1$ elements create with the $k$ elements an extended poly-Cauchy permutation with $x$ boxes. 
\end{proof}
\begin{theorem}[\cite{KimKim}, Eq. (33)]
\begin{align}
	\hat{c}_{n,k}(x+y) = \sum_{j=0}^n\binom{n}{j}\hat{c}_{j,k}(x)(y)_{n-j}
\end{align}
\end{theorem}
\begin{proof}
	The left hand side counts extended poly-Cauchy permutations with $x+y$ boxes. We obtain such an object by choosing $j$ elements out of $[n]$ from that we construct an extended poly-Cauchy permutation with $x$ boxes, and we put the remaining $n-j$ elements into the remaining $y$ boxes. This construction corresponds to the right hand side.
\end{proof}

\section{Other generalizations}
In this section we recall some of the generalizations of poly-Bernoulli numbers and possible extensions of our combinatorial model in these directions.
\subsection{Shifted poly-Cauchy numbers}
Let $\alpha$ be a positive real number. The shifted poly-Cauchy numbers of the second kind have the explicit formula (\cite{Szalay})
\begin{align*}
	\hat{c}_{n, \alpha}^k:= (-1)^{n}\sum_{m=0}^{n}\stf{n}{m} \frac{1}{(m+\alpha)^k}.
\end{align*}
Let us define the objects \emph{augmented poly-Cauchy permutations}. Such a permutation is a poly-Cauchy permutation ending with a red element augmented with $\alpha$ boxes at the end. Into the boxes we can place some blue elements, boxes may be empty. An example with $\alpha = 3$ is given below. The boxes are indicated with $\alpha-1$ bars. 
\begin{align*}
	\blue{49}\red{612}\blue{6}\red{748}\blue{13}\red{953}||\blue{27}|\blue{58}.
\end{align*}   
\begin{theorem}
The	$(n,k)$ augmented poly-Cauchy permutations are enumerated by the shifted poly-Cauchy numbers, $\hat{c}_{n,k, \alpha}=\sum_{m=0}^{n}\stf{n}{m}(m+\alpha)^k$.
\end{theorem}
\begin{proof}
	The difference to the base case is that the possible places where blue elements can be inserted are augmented by $\alpha-1$ bars. In a poly-Cauchy permutation blue elements can be inserted between (before or after) red cycles. Hence, there are now $m+\alpha$ possibilities, when we consider a red permutation with $m$ cycles. 
\end{proof}
A closed formula of the shifted poly-Cauchy numbers, the generalization of Theorem \ref{theo:closed}, can be shown combinatorially as follow.
\begin{theorem}[\cite{Szalay}, Theorem 10]
	For $n$, $k$, and $\alpha$ non-negative integers it holds
	\begin{align*}
		\hat{c}_{n,k,\alpha} = \sum_{i=0}^k\sum_{j=0}^i j! \stf{n+1}{j+1} \binom{k}{i}\sts{i}{j}\alpha^{k-i}.
	\end{align*}
\end{theorem}
\begin{proof}
	Let $k-i$ be the number of blue elements in the boxes. We choose these elements in $\binom{k}{i}$ ways and put them into the boxes in $\alpha^{k-i}$ ways. To the left of the boxes there is a poly-Cauchy permutation on the red elements $[n]$, and blue elements $[k]$, such that the last element is red. The argument of Theorem \ref{theo:closed} has to be a slightly modified to take the type of the last element into account. Namely, we do not need the extra block, so we do not extend the blue element set now with $\blue{*}$. We take a partition of the $i$ blue elements into $j$ blocks in  $\sts{i}{j}$ ways and insert each block directly \emph{before} the chosen cycles.  
\end{proof}
\subsection{Incomplete poly-Cauchy numbers of the second kind}
The  incomplete poly-Cauchy numbers of the second kind have the explicit formula \cite{Kom8}
\begin{align*}
	\hat{c}_{n,\leq m}^k:= (-1)^{n}\sum_{i=0}^{n}\stf{n}{i}_{\leq m} \frac{1}{(i+1)^k},
\end{align*} 
where $\stf{n}{k}_{\leq m}$ denotes the restricted Stirling numbers of the first kind. This sequence counts the number of permutations on  $n$ elements with $k$ cycles that each cycle  contains at  most $m$ items.  From a similar argument we have that $\hat{c}_{n,k, \leq m}=\sum_{i=0}^{n}\stf{n}{i}_{\leq m}(i+1)^k$ counts the number of poly-Cauchy permutations of $[n+k]$, such that the size of each cycle in the cycle decomposition of the  permutation  whose support belongs to  $\{1,2,\ldots, n\}$ is at most $m$.

\begin{example}
	For $n=3$ and $k=1$, $m=2$ we have the following restricted poly-Cauchy permutations:
	\begin{align*}
		&\red{123}\blue{1}, \quad \red{12}\blue{1}\red{3}, \quad \red{132}\blue{1}, \quad  \red{1}\blue{1}\red{23}, \quad  \red{1}\blue{1}\red{32}, \quad  \red{213}\blue{1}, \quad \red{21}\blue{1}\red{3}\\
		&\red{231}\blue{1},  \quad \red{2}\blue{1}\red{31},\quad \blue{1}\red{123}, \quad \blue{1}\red{132}, \quad \blue{1}\red{213},\quad \blue{1}\red{231}
	\end{align*}
	So, $\hat{c}_{3,1, \leq 2}=13$.
	
\end{example}

The  associated poly-Cauchy numbers of the second kind have the explicit formula 
\begin{align*}
	\hat{c}_{n,\leq m}^k:= (-1)^{n}\sum_{i=0}^{n}\stf{n}{i}_{\geq m} \frac{1}{(i+1)^k},
\end{align*}
where $\stf{n}{k}_{\geq m}$ denotes the associated Stirling numbers of the first kind. This sequence counts the number of permutations on  $n$ elements with $k$ cycles with the restriction that each cycle  contains at  least $m$ items. So, the sequence $\hat{c}_{n,k, \geq m}:=\sum_{i=0}^{n}\stf{n}{i}_{\geq m}(i+1)^k$ counts the number of poly-Cauchy permutations of $[n+k]$, such that the size of each cycle in the cycle decomposition of the  permutation   whose support belongs to  $\{1,2,\ldots, n\}$ is at least $m$.

\begin{example}
	For $n=3$ and $k=1$, $m=2$ we have the permutations
	\begin{align*}
		&\red{312}\blue{1}, \quad \red{321}\blue{1}, \quad \blue{1}\red{312}, \quad  \blue{1}\red{321}.
	\end{align*}
	So, $\hat{c}_{3,1, \geq 2}=4$.
	
\end{example}

\subsection{$q$-poly-Cauchy numbers} The $q$-poly-Cauchy numbers of the second kind have the explicit formula \cite{Kom5}
\begin{align*}
	\hat{c}_{n,q}^{(k)}:= (-1)^{n}\sum_{m=0}^{n}\stf{n}{m} \frac{q^{n-m}}{(m+1)^k}.
\end{align*}

It is clear that the sequence $\hat{c}_{n,k}(q):=\sum_{m=0}^{n}\stf{n}{m} q^{n-m}(m+1)^k$ counts the number of poly-Cauchy permutations of $[n+k]$ such that the non left-to-right maxima of $\pi^n$ are colored by one of $q$ colors. 

\begin{example}
	For $n=2$ and $k=2$ we have the permutations
	\begin{align*}
		&\blue{12}\red{12} \to 1,\quad \blue{12}\red{21}\to q, \quad \blue{1}\red{12}\blue{2} \to 1 ,\quad  \blue{1}\red{21}\blue{2} \to q ,\quad \blue{2}\red{12}\blue{1} \to 1, \quad  \blue{2}\red{21}\blue{1} \to q, \quad \blue{1}\red{1}\blue{2}\red{2} \to 1\\
		&\blue{2}\red{1}\blue{1}\red{2} \to 1,  \quad \red{12}\blue{12} \to 1, \quad \red{21}\blue{12}\to q, \quad \red{1}\blue{12}\red{2}\to 1, \quad \red{1}\blue{1}\red{2}\blue{2}\to 1,\quad \red{1}\blue{2}\red{2}\blue{1}\to 1
	\end{align*}
	So, $\hat{c}_{2,2}(q)=4q+9.$
	
\end{example}
\subsection{$q$-poly-Cauchy permutations with a parameter}
There is another generalization, $q$-poly-Cauchy numbers with a parameter \cite{Kom4}, given by the explicit formula
\begin{align*}
	\hat{c}_{n,\rho, q}^{(k)}= (-1)^{n}\sum_{m=0}^{n}\stf{n}{m} \frac{\rho^{n-m}}{[m+1]_q^k},
\end{align*}
 whith
$[n]_q=1+q+\cdots  + q^{n-1}$.
The expression with negative $k$ and ignoring the sign, defined as $\hat{c}_{n, k,\rho, q} = \sum_{m=0}^{n}\stf{n}{m} \rho^{n-m}[m+1]_q^k$ 
counts poly-Cauchy permutations with some weight on both the red elements and blue elements. As in the case of $q$-poly-Cauchy permutations we can color the  non left-to-right maxima of $\pi^n$ by one of $\rho$ colors. This notion is extended by a weight on each blue elements, $b$. Namely, $w(b) = q^{\#\mbox{red cycles to the left of $b$}}$. The weight of the poly-Cauchy permutation is the product of weights of all the elements.
Formally, let $\mathcal{P}_{n,k}$ denote the set of poly-Cauchy permutations of $[n+k]$, and $\mu(\pi)$ the number of non left-to-right maxima in $\pi^n$.  Then we have
\begin{align*}
		\hat{c}_{n,k,\rho, q} = \sum_{\pi\in \mathcal{P}}\rho^{\mu(\pi)}\prod_{i=1}^{k}q^{w(n+i)}. 
\end{align*}

\begin{example}
	For $n=2$ and $k=2$  the weight of the permutations are
	\begin{align*}
		&\blue{12}\red{12} \to 1,\quad \blue{12}\red{21}\to \rho, \quad \blue{1}\red{12}\blue{2} \to q^2 ,\quad  \blue{1}\red{21}\blue{2} \to q\rho ,\quad \blue{2}\red{12}\blue{1} \to q^2, \quad  \blue{2}\red{21}\blue{1} \to q\rho, \quad \blue{1}\red{1}\blue{2}\red{2} \to q\\
		&\blue{2}\red{1}\blue{1}\red{2} \to q,  \quad \red{12}\blue{12} \to q^4, \quad \red{21}\blue{12}\to q^2\rho, \quad \red{1}\blue{12}\red{2}\to q^2, \quad \red{1}\blue{1}\red{2}\blue{2}\to q^3,\quad \red{1}\blue{2}\red{2}\blue{1}\to q^3.
	\end{align*}
	So, $\hat{c}_{n,k,\rho,q} = \rho (1+2q+q^2)+(1+2q+3q^2+2q^3+q^4)$.
	
\end{example}

\section{Conclusion}
In this paper we presented a combinatorial interpretation of poly-Cauchy numbers of the second kind. We showed how this model can be extended to some of the generalizations that were introduced in the literature on these numbers. However, this paper on this topic is not even aimed to be complete. One of our goal was to show the power of combinatorial thinking, proofs with pictures.

In this section we list some open questions that would be interesting to understand combinatorially.  
\begin{itemize}
\item[1.] We studied in this work only the poly-Cauchy numbers (and their generalizations) of the second kind. The question arises naturally, is it possible to give a similar interpretation of the poly-Cauchy of the first kind?
The relation between the two kinds of poly-Cauchy numbers is given for instance by the following formula:
\begin{align*}
(-1)^n\frac{c_n^k}{n!}= \sum_{m=1}^{n}\binom{n-1}{m-1} \frac{\hat{c}_m^k}{m!}.
\end{align*}
\item[2.]We did not present a combinatorial proof for all of the formulas that are known, though we think most of the identities could be attacked using our combinatorial interpretation. The interested reader can try to find combinatorial proof for them. For example, recently, Komatsu \cite{Kom6} (Theorem 8) showed an annihilation formula
\begin{align*}
	\sum_{\ell = 0}^k\sts{n+1}{n-\ell+1}_{n-k}\hat{c}_{n-\ell}^{(-k)} =0.
\end{align*}
\item[3.] There are also several interesting identities involving generalizations of the poly-Cauchy numbers that have a combinatorial ``flavour''. As an example, we recall the expression that was given by Komatsu and Szalay \cite{Szalay} in their study about the shifted poly-Cauchy numbers. 
\begin{theorem}\cite{Szalay}
	\begin{align*}
	\hat{c}_{n,\alpha}^{(-k)} = (-1)^{\alpha-1}\sum_{\mu=0}^{\alpha-1}Q_{\mu}(n,\alpha)\hat{c}_{n+\mu}^{(-k)}, 
	\end{align*}
where 
\begin{align*}
	Q_{\mu}(n,\alpha) = \sum_{i=0}^{\alpha-\mu-1} \binom{\alpha-1}{i}\sts{\alpha-i-1}{\mu}n^i, \quad (0\leq \mu\leq \alpha-1).
\end{align*}
\end{theorem}
\item[4.] There are even more generalizations in the literature, as multi-poly-Cauchy numbers and polynomials (cf. \cite{KimKim2}). 
\item[5.] It would be of course very interesting to find other combinatorial interpretations related to these numbers. 
\end{itemize}

\end{document}